\documentclass[12pt]{amsart}
\usepackage{amsthm}
\usepackage{amsmath}
\usepackage{amssymb}
\usepackage[backend=biber,uniquename=init]{biblatex}
\bibliography{references.bib}
\AtEveryBibitem{
  \clearfield{doi}
  \clearfield{isbn}
  \clearfield{issn}
  \clearfield{url}
  \clearfield{MRNUMBER}
}

\usepackage{mathrsfs}
\usepackage{enumerate}
\usepackage[labelfont=bf]{caption}
\usepackage{enumitem}
\usepackage{booktabs}
\usepackage{algorithm,setspace}
\usepackage{algpseudocode}
\usepackage{float}
\usepackage{pgf}
\usepackage{url, multicol}
\usepackage[all]{xy}
\usepackage{graphicx}
\usepackage{nicematrix}
\usepackage{hyperref}
\usepackage{pgfplots}
\usepackage[noabbrev,nosort,capitalise]{cleveref}

\usepackage[margin = 1 in]{geometry}
\usepackage{tikz}
\usetikzlibrary{
  arrows.meta,shapes,automata,petri,positioning,calc,
  graphs,
  graphs.standard
}

\definecolor{royalblue}{HTML}{59B5F7}

\newcommand{\Z}{{\mathbb Z}}

\DeclareMathOperator{\Sat}{Sat}

\DeclareMathOperator{\Ex}{Ex}

\newtheorem{thm}{Theorem}
\crefname{thm}{Theorem}{Theorems}
\newtheorem{conj}{Conjecture}
\crefname{conj}{Conjecture}{Conjectures}
\newtheorem{prop}[thm]{Proposition}

\theoremstyle{definition}
\newtheorem{rem}{Remark}

\newtheorem{exam}{Example}
\numberwithin{equation}{section}
\numberwithin{table}{thm}

\allowdisplaybreaks

\title{Upper Bounds for Sequence Saturation}
\author[Shihan Kanungo]{Shihan Kanungo}
\address{San Jos\'e State University \ San Jos\'e, CA 95192}
\email{shihan.kanungo@sjsu.edu}

\date{\today}   

\subjclass[2020]{05D99}
\keywords{Sequence Saturation, Generalized Davenport-Schinzel Sequences}
\pgfplotsset{compat=1.18}

\begin{document}
\begin{abstract}
    In this paper, we study the saturation function $\Sat(n,u)$ for sequences. Saturation for sequences was introduced by Anand, Geneson, Kaustav, and Tsai (2021), who proved that $\Sat(n,u)=O(n)$ for two-letter sequences $u$ and conjectured that this bound holds for all sequences. We present an algorithm that constructs a $u$-saturated sequence on $n$ letters and apply it to show $\Sat(n,u)=O(n)$ for several families of sequences $u$, including all repetitions of the form $abcabc\dots$. We further establish $\Sat(n,u)=O(n)$ for a broad class of sequences of the form $aa\dots bb$. In addition, we prove that for most sequences $u$, there exists an infinite $u$-saturated sequence. For three-letter sequences of the form $abc\dots xyz$, where $a,b,c$ are distinct and $xyz$ is a permutation of $abc$, we show---under certain structural assumptions on $u$---that $\Sat(n,u)=O(n)$. Finally, we describe a linear program that computes the exact value of $\Sat(n,u)$ for arbitrary $n$ and $u$.
\end{abstract}
\maketitle

\section{Introduction}
The concept of pattern avoidance lies at the heart of extremal combinatorics, which investigates questions such as: \vskip 5pt
\begin{center}
    \textit{What is the largest possible structure that avoids a given substructure}?
\end{center}  \vskip 5pt
Well-known examples include Ramsey theory and the forbidden subgraph problem. Another important example is the study of \textbf{Davenport-Schinzel sequences}, which are sequences $S$ over an alphabet of $n$ distinct letters that satisfy the following constraints:
\begin{itemize}[topsep=5pt,itemsep=3pt]
    \item For any two distinct letters $a$, $b$, the alternating subsequence  $abab\dots$  of length \newline $s+2$ is forbidden, where $s$ is a nonnegative integer.
    \item No two consecutive letters in $S$ are the same.
\end{itemize}

A key quantity of interest is the function $\lambda_s(n)$, which denotes the length of the longest Davenport-Schinzel sequence over an $n$-letter alphabet. For $s\ge 3$, $\lambda_s(n)$ grows as $n$ times a small, but nonconstant, factor. For instance, \[\lambda_3(n)=\Theta(n\alpha(n)),\] where $\alpha(n)$ is the inverse Ackermann function\footnote{see \cite{PETTIE20111863} for a definition}, known for its extremely slow growth. Originally introduced to analyze linear differential equations (\cite{davenportschinzel65}), Davenport-Schinzel sequences and the function $\lambda_s(n)$ have since found numerous applications in discrete geometry and geometric algorithms (\cite{Atallah85}), as well as application to the saturation function for $0$--$1$ matrices (\cite{Keszegh09}).

A \textbf{generalized Davenport-Schinzel sequence} extends this concept by requiring sequences to avoid a given pattern $u$ on $r$ distinct letters while also being $r$\textbf{-sparse}, meaning every contiguous subsequence
of length $r$ has all letters distinct. Analogous to $\lambda_s(n)$, we define the \textbf{extremal function} $\Ex(n,u)$, which represents the longest $r$-sparse sequence on $n$ letters that avoids $u$. The growth of $\Ex(n,u)$ has been studied in \cite{PETTIE20111863}, and generalized Davenport-Schinzel sequences have many significant applications.

This paper focuses on a related function, the \textbf{saturation function} $\Sat(n,u)$. Saturation functions have been extensively studied in other combinatorial settings, including graphs \cites{KászonyiTuzaGraphSaturation,DAMASDI2021103321}, posets \cite{Ferrara2017TheSN}, and 0-1 matrices \cite{Fulek2020SaturationPA}. The notion of saturation for sequences was introduced in \cite{ANAND2025382}, where several fundamental results were also established. Specifically, given a forbidden sequence $u$ with $r$ distinct letters, we say that a sequence $s$ on a given alphabet is \textbf{$u$-saturated} if $s$ is $r$-sparse, $u$-free, and adding any letter from the alphabet to an arbitrary position in $s$ violates $r$-sparsity or induces a copy of $u$. Notably, it was shown that if $u$ consists of only two distinct letters, then $\Sat(n,u)=O(1)$ or $\Theta(n)$, leading to the open question of whether this dichotomy extends to arbitrary sequences $u$. In fact, proving that $\Sat(n,u)=O(n)$ for all sequences $u$ is sufficient to establish this dichotomy. Thus, we aim to address the following conjecture:

\begin{conj}[\cite{ANAND2025382}]\label{conj: main}
    We have $\Sat(n,u) = O(n)$ for any sequence $u$.
\end{conj}
Similar dichotomies have been proven in other settings, such as for graphs \cite{KászonyiTuzaGraphSaturation} and for 0-1 matrices \cite{Fulek2020SaturationPA}. However, proving the $O(n)$ bound in the context of sequence saturation appears to be significantly more challenging. 

We resolve \cref{conj: main} in the following cases.

For a sequence $u= abc\ldots xyz$ with 3 distinct letters $a,b,c$, let
\begin{align*}
    f_0(u) &= \#\{\,\text{consecutive pairs of the form } ab,bc,ca\text{ in $u$}\}, \\
    f_1(u) &= \#\{\,\text{consecutive pairs of the form } ac,ba,cb\text{ in $u$}\}, \\
    f_2(u) &= \#\{\,\text{consecutive equal-letter pairs in } u\,\}.
\end{align*}

\begin{thm}\label{thm: main}
    \cref{conj: main} holds for the following classes of sequences $u$:
    \begin{enumerate}
        \item $u=abcabc\cdots$,
        \item $u=aa\dots bb$ and $u$ cannot be decomposed into two subsequences $u=u_1u_2$ that have no letters in common,
        \item $u=abc\dots xyz$ is a three-letter sequence with $a,b,c$ distinct, and
        \[xyz\in \{abc,bca,cab\}, \qquad f_0(u)\ge f_1(u)+5.\]
    \end{enumerate}
\end{thm}

The organization of this paper is as follows. In \cref{sec: defns} we define the saturation and extremal functions for sequences. In \cref{sec: saturation}, we introduce \cref{alg: algorithm}, which produces a $u$-saturated sequence on $n$ letters. Using this, we give an example of a sequence $u$ with $\Sat(n,u)=O(n)$ even though $\Ex(n,u)$ grows faster than linear, and we additionally prove point (1) of \cref{thm: main}. 
We next prove point (2) of \cref{thm: main} in \cref{thm: aa...bb}. Finally, we show in \cref{thm: infinite sequences} that for a large class of sequences $u$, there exists a doubly infinite sequence that is $u$-saturated. In \cref{sec: 3-letter sequences}, we focus on sequences $u$ with $r=3$ letters, and prove point (3) of \cref{thm: main} in \cref{thm: 3-letter sequences}. In \cref{sec: linear program}, we introduce a linear program in \cref{thm: linear program} that computes the exact value of the $\Sat(n,u)$, similarly to the one for 0-1 matrices in \cite{brahms01matrices}. In \cref{tab: linear program}, we display the results of the linear program for some short sequences $u$ and small $n$. 
 
A further direction to investigate, after \cref{conj: main} is resolved, would be to classify when $\Sat(n,u) = O(1)$ or $\Theta(n)$. The corresponding question for 0-1 matrices was investigated by Geneson \cite{Geneson2020AlmostAP}, who showed that almost all permutation matrices have bounded saturation function, and by Berendsohn \cite{Berendsohn2021AnEC}, who completely resolved the classification for permutation matrices. However, the general question remains open.

\section{Notation and Definitions}\label{sec: defns}
For a sequence $s$, let $s^t$ denote the concatenation of $t$ copies of $s$. Let $[n]$ denote the set $\{1,\dots, n\}$.

Throughout this paper we consider a finite sequence $u$ that contains $r$ distinct letters. We say sequences $u$ and $v$ are \textbf{isomorphic} if $v$ can be obtained from $u$ by a one-to-one renaming of letters. We say a sequence $s$ \textbf{avoids} $u$ if $s$ does not contain a subsequence isomorphic to $u$. We refer to any subsequence of $s$ isomorphic to $u$ as a \textbf{copy} of $u$ in $s$. We say that $s$ is \textbf{$r$-sparse} if any $r$ consecutive letters in $s$ are pairwise distinct. We say that a sequence $s$ with letters in an alphabet $A$ is \textbf{$u$-saturated} if $s$ is $r$-sparse, $s$ avoids $u$, and inserting any letter in $A$ between two consecutive letters in $s$ either causes $s$ to not be $r$-sparse or causes $s$ to contain a copy of $u$. Finally, we say that $s$ is \textbf{$u$-semisaturated} if $s$ is $r$-sparse and inserting a letter either violates $r$-sparsity or induces a copy of $u$. The notions of saturation and semisaturation are the same, except for the fact that saturated sequences have to avoid $u$.

We define the \textbf{extremal function} $\Ex(n,u)$ to be the maximum length of a sequence $s$ on $n$ letters such that $s$ avoids $u$ and is $r$-sparse. We define the \textbf{saturation function} $\Sat(n,u)$ to be the minimum length of a $u$-saturated sequence $s$ on $n$ letters. Trivially, we have $\Sat(n,u)\le \Ex(n,u)$ because every $u$-saturated sequence is $r$-sparse and avoids $u$.

A \textbf{0-1 matrix} is a matrix whose entries are $0$ or $1$. By standard convention, we represent 0-1 matrices by replacing the ones with bullets $\bullet$ and the zeros with empty spaces. For example, we write
\[\begin{pmatrix}
    \bullet & \bullet\\
    \bullet & 
\end{pmatrix}\quad \text{for}\quad \begin{pmatrix}
    1 & 1 \\
    1 & 0
\end{pmatrix}.\]
We say that a 0-1 matrix $A$ \textbf{contains} another 0-1 matrix $B$ if $A$ has a submatrix with the same dimensions as $B$ that is entrywise greater than $B$.

\section{General Saturation of Sequences}\label{sec: saturation}
We are interested in showing that $ \Sat(n,u) = O(n) $ for any sequence $ u $. This has been shown for $ 2 $-letter sequences $ u $ in \cite{ANAND2025382}, but the cases for $ 3 $-letter sequences and higher remain open.

\subsection{Algorithm}
In \cref{alg: algorithm} (see below), we describe a method to construct a $u$-saturated sequence on $n$ letters. We say ``$x$ can be properly inserted into $s$'' as a shorthand to mean ``$x$ can be inserted into $s$ without violating $r$-sparsity and without inducing a copy of $u$''. It is evident from the definition of the algorithm that it will always produce a $u$-saturated sequence. To demonstrate that $ \Sat(n,u) = O(n) $, we must construct a $ u $-saturated sequence on $ n $ letters for every $ n $, where the length of the sequence grows at most linearly with $ n $. In this subsection, we use \cref{alg: algorithm} to accomplish this task for a variety of sequences $u$.

\begin{algorithm}
\setstretch{1.35}
\begin{algorithmic}[1]
\State \textbf{Input:} Alphabet $A = \{1, \dots, n\}$, forbidden sequence $u$
\State \textbf{Output:} $u$-saturated sequence
\State Initialize the sequence: $s \gets 1, 2, \dots, r-1$ \Comment{Initial sequence avoids $u$}
\While{it is possible to extend the sequence}
    \For{each letter $x \in A$}
        \If{$x$ can be \textit{properly inserted} into $s$}
            \State Insert $x$ appropriately into $s$ to form $s'$ \Comment{Smallest $x$, leftmost position}
            \State Update $s \gets s'$ \Comment{New sequence}
            \State \textbf{break} \Comment{Exit loop after the first valid insertion}
        \EndIf
    \EndFor
\EndWhile
\State \textbf{Return} $s$ \Comment{Final sequence is $u$-saturated}
\end{algorithmic}
\caption{Constructing a $u$-Saturated Sequence}\label{alg: algorithm}
\end{algorithm}

To represent long sequences, we use the following method. Given a sequence $s = s_1 \dots s_k$, we plot the points $(i, s_i)$ for $i = 1, \dots, n$. For instance, the sequence $1, 2, \dots, n$ forms a line with a slope of $1$, while the sequence $1, 1, \dots, 1$ results in a horizontal line. 

Trivially, we know that $ \Sat(n,u) \leq \Ex(n,u) $, so if $ \Ex(n,u) = O(n) $, \cref{conj: main} immediately follows. A \textbf{nonlinear sequence} is defined as one for which $ \Ex(n,u) $ grows faster than linearly. Therefore, we only need to focus on nonlinear sequences.

In \cite{PETTIE20111863}, the extremal function for sequences, particularly $3$-letter sequences, was studied. It was shown that the sequence $ u = abcacbc $ is nonlinear and ``minimally nonlinear,'' meaning that no subsequence of $ u $ is nonlinear. Thus, $u=abcacbc$ is the first sequence that we will consider. If we run \cref{alg: algorithm} on $u=abcacbc$, we notice that a pattern appears. In \cref{fig: abcacbc}, we display the output of the algorithm using the representation described earlier.

\begin{figure}[ht]
\centering
\begin{tikzpicture}
\begin{axis}[
    width=12cm,
    height=8cm,
    line width=1pt,
    font=\footnotesize,
    xmin=0,
    xmax=42,
    grid=none,
]
\addplot[
    only marks,
    mark=*,
    color=black!55
]
coordinates {
(1,1) (2,2) (3,10) (4,1) (5,2) (6,9)
(7,1) (8,2) (9,8) (10,1) (11,2) (12,7)
(13,1) (14,2) (15,6) (16,1) (17,2) (18,5)
(19,1) (20,2) (21,4) (22,1) (23,2) (24,3)
(25,1) (26,2) (27,3) (28,1) (29,2) (30,10)
(31,9) (32,8) (33,7) (34,6) (35,5) (36,4)
(37,6) (38,7) (39,8) (40,9) (41,10)
};
\end{axis}
\end{tikzpicture}
\caption{\cref{alg: algorithm} on $u=abcacbc$.}\label{fig: abcacbc}
\end{figure}

This example is the $n=10$ case. The pattern (which is easily seen from the figure) corresponds to the sequence
\[1,2,n,1,2,n-1,\dots,1,2,3,1,2,3,1,2,n,n-1,\dots,6,5,4,6,7,\dots,n.\]
From here it is straightforward (albeit quite tedious) to verify that this is a $u$-saturated sequence, for any $n$. Since the length of this sequence grows linearly with $n$, it follows that
\begin{prop}
    We have $\Sat(n,abcacbc)=O(n)$.
\end{prop} 

We can do this for different sequences. For every case we tried, we got a repeating pattern like the one above. Some of these had minor variations that depended on the value of $n\pmod{d}$ for some $d$; but the $O(n)$ bound still follows. Some more examples are shown in \cref{fig: more examples}.

\medskip

\begin{figure}[ht]
\begin{tikzpicture}
\begin{axis}[
    width=8cm,
    height=6cm,
    line width=1pt,
    font=\footnotesize,
    xmin=0,
    xmax=39,
    grid=none,
]
\addplot[
    only marks,
    mark=*,
    color=black!55
]
coordinates {
(1, 10)
(2, 9)
(3, 8)
(4, 10)
(5, 7)
(6, 9)
(7, 6)
(8, 8)
(9, 10)
(10, 5)
(11, 7)
(12, 9)
(13, 4)
(14, 6)
(15, 8)
(16, 2)
(17, 5)
(18, 7)
(19, 3)
(20, 4)
(21, 6)
(22, 1)
(23, 2)
(24, 5)
(25, 3)
(26, 4)
(27, 1)
(28, 2)
(29, 3)
(30, 1)
(31, 2)
(32, 4)
(33, 5)
(34, 6)
(35, 7)
(36, 8)
(37, 9)
(38, 10)
};
\end{axis}
\end{tikzpicture}
\hspace{10mm}
\begin{tikzpicture}
\begin{axis}[
    width=8cm,
    height=6cm,
    line width=1pt,
    font=\scriptsize,
    xmin=0,
    xmax=31,
    grid=none,
]
\addplot[
    only marks,
    mark=*,
    color=black!55
]
coordinates {
(1, 10)
(2, 9)
(3, 8)
(4, 10)
(5, 7)
(6, 9)
(7, 6)
(8, 8)
(9, 10)
(10, 5)
(11, 7)
(12, 9)
(13, 4)
(14, 6)
(15, 8)
(16, 3)
(17, 5)
(18, 7)
(19, 1)
(20, 4)
(21, 6)
(22, 2)
(23, 3)
(24, 5)
(25, 1)
(26, 4)
(27, 2)
(28, 3)
(29, 1)
(30, 2)
};
\end{axis}
\end{tikzpicture}
\caption{\cref{alg: algorithm} on $u=abbacac$ (left) and $abcacba$ (right).}\label{fig: more examples}
\end{figure}
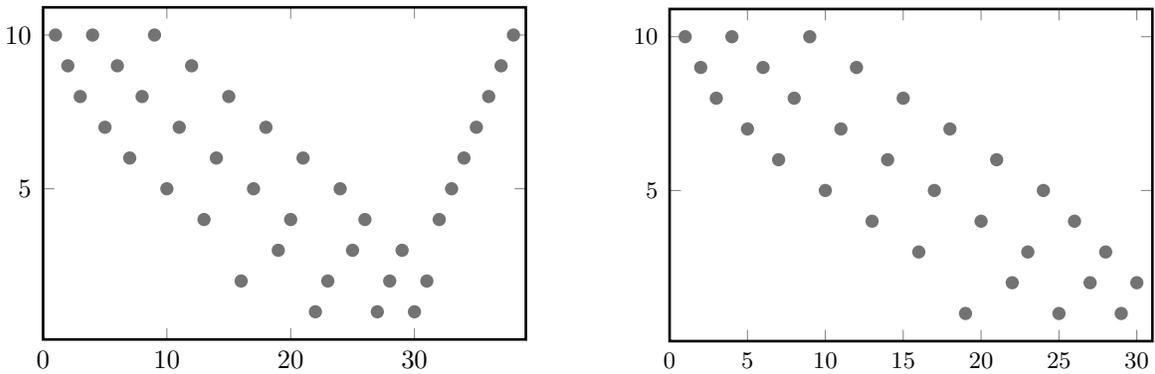

The fact that we always get a pattern prompts us to posit the following.
\begin{conj}\label{conj: algorithm}
    Let $s(n,u)$ be the sequence of produced by \cref{alg: algorithm}. Then the length of $s(n,u)$ grows at most linearly in $n$, for any $u$.
\end{conj}
There is one important point to note about this algorithm: since we prefer to insert smaller letters, in the process of generating $s(n,u)$, right before we add the first ``$n$'', the sequence must be exactly $s(n-1,u)$. Thus, $s(n-1,u)$ is a subsequence of $s(n,u)$. This means that \cref{conj: algorithm} is equivalent to saying that if we run the algorithm starting from $s(n-1,u)$, we will add a bounded number of new letters.

We now use this algorithm to prove the conjecture when $u$ is of the form $abcabc\dots$. Such sequences are the three-letter analog of the Davenport-Schinzel sequences. In the case when $u = (abc)^t$, for some integer $t$, the algorithm outputs a sequence of the form shown in \cref{fig: (abc)^t}.
\begin{figure}[ht]
\centering
\begin{tikzpicture}
\begin{axis}[
    width=12cm,
    height=8cm,
    line width=1pt,
    font=\footnotesize,
    xmin=0,
    xmax=75,
    grid=none,
]
\addplot[
    only marks,
    mark=*,
    color=black!55
]
coordinates {
(1, 1)
(2, 2)
(3, 10)
(4, 1)
(5, 2)
(6, 10)
(7, 1)
(8, 2)
(9, 10)
(10, 1)
(11, 2)
(12, 9)
(13, 1)
(14, 2)
(15, 9)
(16, 1)
(17, 2)
(18, 9)
(19, 1)
(20, 2)
(21, 8)
(22, 1)
(23, 2)
(24, 8)
(25, 1)
(26, 2)
(27, 8)
(28, 1)
(29, 2)
(30, 7)
(31, 1)
(32, 2)
(33, 7)
(34, 1)
(35, 2)
(36, 7)
(37, 1)
(38, 2)
(39, 6)
(40, 1)
(41, 2)
(42, 6)
(43, 1)
(44, 2)
(45, 6)
(46, 1)
(47, 2)
(48, 5)
(49, 1)
(50, 2)
(51, 5)
(52, 1)
(53, 2)
(54, 5)
(55, 1)
(56, 2)
(57, 4)
(58, 1)
(59, 2)
(60, 4)
(61, 1)
(62, 2)
(63, 4)
(64, 1)
(65, 2)
(66, 3)
(67, 1)
(68, 2)
(69, 3)
(70, 1)
(71, 2)
(72, 3)
(73, 1)
(74, 2)
};
\end{axis}
\end{tikzpicture}
\caption{\cref{alg: algorithm} on $u=(abc)^t$, $t=4$.}\label{fig: (abc)^t}
\end{figure}
This example is when $t=4,n=10$. In general, the numbers $3,\dots, n$ will each appear exactly $t-1$ times. The sequence obtained is thus
\[(12n)^{t-1}\dots (124)^{t-1}(123)^{t-1}12.\]
It is straightforward to check that the resulting sequence is $(abc)^t$-saturated. Note that this is a special case of Lemma 3.4 in \cite{ANAND2025382}.

Similarly, when $u=(abc)^ta$ we get the pattern in \cref{fig: (abc)^ta}.
\begin{figure}[ht]
\centering
\begin{tikzpicture}
\begin{axis}[
    width=12cm,
    height=8cm,
    line width=1pt,
    font=\footnotesize,
    xmin=0,
    xmax=83,
    grid=none,
]
\addplot[
    only marks,
    mark=*,
    color=black!55
]
coordinates {
(1, 10)
(2, 9)
(3, 8)
(4, 10)
(5, 9)
(6, 8)
(7, 10)
(8, 9)
(9, 8)
(10, 10)
(11, 9)
(12, 8)
(13, 7)
(14, 9)
(15, 8)
(16, 7)
(17, 9)
(18, 8)
(19, 7)
(20, 9)
(21, 8)
(22, 7)
(23, 6)
(24, 8)
(25, 7)
(26, 6)
(27, 8)
(28, 7)
(29, 6)
(30, 8)
(31, 7)
(32, 6)
(33, 5)
(34, 7)
(35, 6)
(36, 5)
(37, 7)
(38, 6)
(39, 5)
(40, 7)
(41, 6)
(42, 5)
(43, 4)
(44, 6)
(45, 5)
(46, 4)
(47, 6)
(48, 5)
(49, 4)
(50, 6)
(51, 5)
(52, 4)
(53, 3)
(54, 5)
(55, 4)
(56, 3)
(57, 5)
(58, 4)
(59, 3)
(60, 5)
(61, 4)
(62, 3)
(63, 1)
(64, 4)
(65, 3)
(66, 1)
(67, 4)
(68, 3)
(69, 1)
(70, 4)
(71, 3)
(72, 1)
(73, 2)
(74, 3)
(75, 1)
(76, 2)
(77, 3)
(78, 1)
(79, 2)
(80, 3)
(81, 1)
(82, 2)
};
\end{axis}
\end{tikzpicture}
\caption{\cref{alg: algorithm} on $u=(abc)^ta$, $t=4$.}\label{fig: (abc)^ta}
\end{figure}
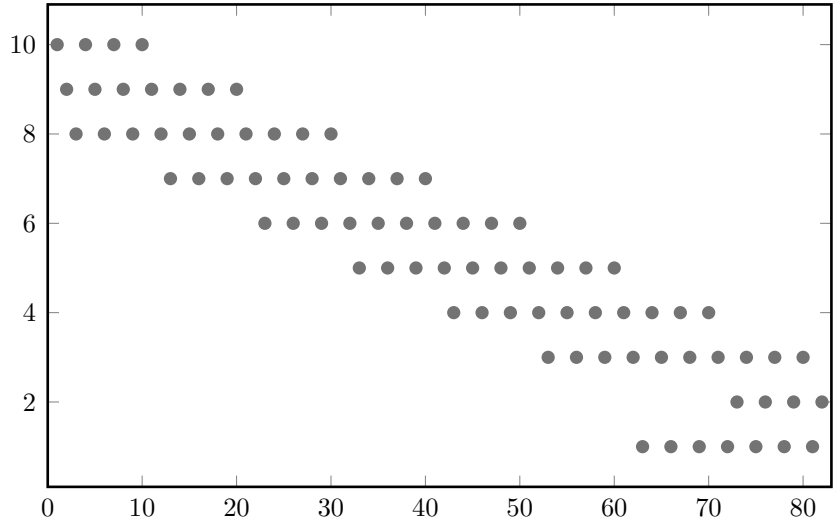
Here $t=4,n=10$, and each number from $3$ to $n-2$ appears exactly $3(t-1)+1$ times. This sequence can be easily generalized to arbitrary values of $t$.
Finally, for $u=(abc)^tab$, we get the pattern in \cref{fig: (abc)^tab}, where $t=3$, $n=10$, and each number from $3$ to $n-2$ appears exactly $3(t-1)+2$ times.
\begin{figure}[ht]
\centering
\begin{tikzpicture}
\begin{axis}[
    width=12cm,
    height=8cm,
    line width=1pt,
    font=\footnotesize,
    xmin=0,
    xmax=67,
    grid=none,
]
\addplot[
    only marks,
    mark=*,
    color=black!55
]
coordinates {
(1, 9)
(2, 10)
(3, 8)
(4, 9)
(5, 10)
(6, 8)
(7, 9)
(8, 10)
(9, 8)
(10, 9)
(11, 7)
(12, 8)
(13, 9)
(14, 7)
(15, 8)
(16, 9)
(17, 7)
(18, 8)
(19, 6)
(20, 7)
(21, 8)
(22, 6)
(23, 7)
(24, 8)
(25, 6)
(26, 7)
(27, 5)
(28, 6)
(29, 7)
(30, 5)
(31, 6)
(32, 7)
(33, 5)
(34, 6)
(35, 4)
(36, 5)
(37, 6)
(38, 4)
(39, 5)
(40, 6)
(41, 4)
(42, 5)
(43, 3)
(44, 4)
(45, 5)
(46, 3)
(47, 4)
(48, 5)
(49, 3)
(50, 4)
(51, 2)
(52, 3)
(53, 4)
(54, 2)
(55, 3)
(56, 4)
(57, 2)
(58, 3)
(59, 1)
(60, 2)
(61, 3)
(62, 1)
(63, 2)
(64, 3)
(65, 1)
(66, 2)
};
\end{axis}
\end{tikzpicture}
\caption{\cref{alg: algorithm} on $u=(abc)^tab$, $t=3$.}\label{fig: (abc)^tab}
\end{figure}
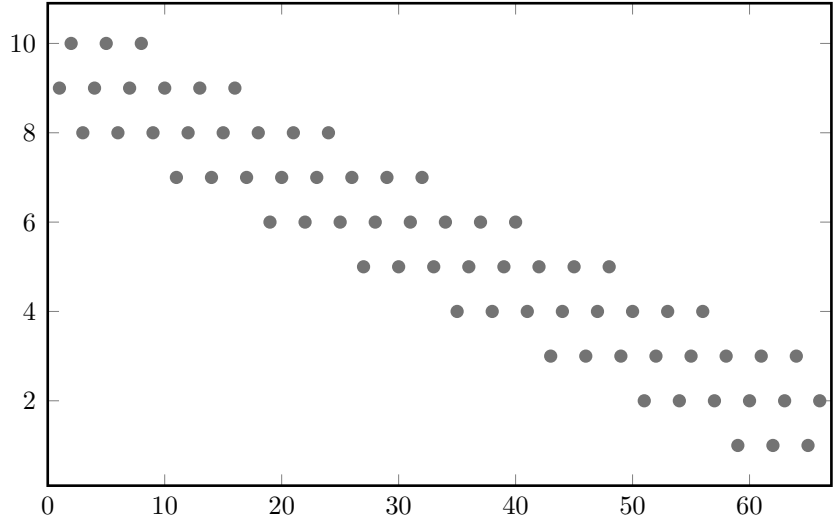
As before, this sequence can be easily generalized to arbitrary values of $t$.

Proving that these constructions indeed give $u$-saturated sequences is straightforward but tedious. In summary, we have:
\begin{prop}\label{prop: generalized alternations}
    If $u$ is of the form $abcabc\dots$, $\Sat(u) = O(n)$.
\end{prop}
This prove point (1) of \cref{thm: main}.
If the length of $u$ is at least $6$, this result also follows from \cref{thm: 3-letter sequences} below. Interestingly, the constructions above are quite similar (or the same) as the construction given in \cref{thm: 3-letter sequences}.

\subsection{Sequences of the form \texorpdfstring{$aa\dots bb$}{aa...bb}}
In this subsection, we prove a result that resolves \cref{conj: main} for a large class of sequences.

Consider a sequence $u$ on $r$ distinct letters. We define $s_r(u)$ to be the longest sequence of the form $a_1a_2\dots a_r a_1 a_2\dots$ that avoids $u$. It is straightforward to verify that $s_r(u)$ is $u$-saturated. Let $s_r^m(u)$ be a copy of $s_r(u)$ on the letters $m+1,m+2,\dots,m+r$. The sequence $s_r(u)$ was first defined in \cite{ANAND2025382} for two-letter sequences, and was a key component in the proof of the $O(n)$ bound for two-letter sequences $u$. 

We say that $u$ is \textbf{irreducible} if $u$ cannot be decomposed into subsequences $u=u_1u_2$ such that $u_1$ and $u_2$ have no letters in common.

\begin{thm}\label{thm: aa...bb}
If $u$ is irreducible and of the form $aa\dots bb$, then $\Sat(n,u)=O(n)$.
\end{thm}

\begin{proof}  
Let $r$ be the length of $u$. If $r = 2$, the result follows from \cite{ANAND2025382}, so assume $r \ge 3$. The idea is to create a sequence out of blocks that are isomorphic to $s_r(u)$ (which is known to be $u$-saturated). In the case where $n$ is not evenly divisible by $r$, we have to slightly modify the lengths of some of the blocks, but the general idea is the same.

\vspace{2mm}\noindent
\textsc{Case 1:} $n = rm$ for some integer $m$.\newline
Define the sequence $s_n$ as follows:
\[s_n = s_r^0(u) \, s_r^r(u) \dots s_r^{r(m-1)}(u),\]
We claim that $s_n$ is $u$-saturated.

\textit{Avoidance:} Partition $[n]$ into \textit{blocks} of size $r$: the $k$-th block is $\{rk+1, \dots, rk + r\}$. Suppose that $s_n$ contains a copy of $u$. By irreducibility of $u$, the letters in this copy all come from the same block. But within each block, the subsequence is $s_r(u)$, which avoids $u$. Hence, $s_n$ avoids $u$.

\textit{Saturation:} Suppose we insert the letter $a$ into $s_n$. Then $a$ belongs to some block $\{rk+1, \dots, rk+r\}$. If we insert $a$ inside the corresponding $s_r^{rk}(u)$, it is easy to see that we break $r$-sparsity. Now suppose we insert $a$ before or after $s_r^{rk}(u)$. Without loss of generality, we insert $a$ to the left of $s_r^0(u)$. Then by the same argument as below in Case 2, we induce a copy of $u$ in $a,s_r^0(u)$.

\vspace{2mm}\noindent
\textsc{Case 2:} ${n = rm + j}$, where ${1 \le j < r}$. \newline
We want to construct a $u$-saturated sequence of length $n$. Let $s_r(u) = (a_1 \dots a_r)^t a_1 \dots a_j$, and define:
\[s_r'(u) = (a_1 \dots a_{r+1})^t a_1 \dots a_j.\]
We claim $s_r'(u)$ is also $u$-saturated and behaves like $s_r(u)$ in the saturation argument above. In particular, it is $u$ saturated, and if we add any letter on the left or the right of $s_r'(u)$ (even if it violates sparsity) then we get a copy of $u$.

\textit{Avoidance:} This is clear, since $s_r(u)$ avoids $u$.

\textit{Saturation:} Note that the occurrences of the letter $a_i$ in $s_r'(u)$ are evenly spaced with exactly $r$ letters between them. Since $r\ge 3$, this means we cannot insert an $a_i$ between two already-existing $a_i$ in $s_r'(u)$ without violating $r$-sparsity. So it has to be inserted before the very first or after the very last $a_i$. Note that this is trivially true also when we are adding $a_i$ on the left or the right of $s_r'(u)$. We claim that as long as this condition is satisfied, we have a copy of $u$.
    
We can take a subsequence $s$ of $s_r'(u)$ containing $a_i$ that is isomorphic to $s_r(u)$ (we just delete all occurrences of one of the letters $a_k$, $k>j$). Then $a_i$ is inserted either before the first $a_i$ in $s$ or after the last. Without loss of generality, it is the first case, and $s$ uses the letters $a_1,\dots, a_{r}$. Then, let $u'$ be $u$ with the first letter deleted. Since the first two letters of $u$ are equal, it is clear that $s_r(u')$ is $s_r(u)$ with the first $r$ letters deleted. Letting $s = a_1\dots a_{i-1}s'$, it follows that $s'$ is isomorphic to a sequence of the form $b_1\dots b_r b_1\dots b_r\dots$ that is longer than $s_r(u')$, hence contains a copy of $u'$. Furthermore, this copy can start with the first letter of $s'$, which is $a_i$. By adding the $a_i$ which we inserted before this $a_i$, we get a copy of $u$. This proves that $s_r'(u)$ is $u$-saturated, and also if we add a letter on the left or the right of $s_r'(u)$, we get a copy of $u$.

Now, to reach length $n = rm + j$, replace $j$ of the blocks $s_r^{rk}(u)$ in $s_{rm}$ with $s_r'(u)$. The resulting sequence has length $n$ and remains $u$-saturated. Since we only use $O(m) = O(n)$ blocks of constant size, the total length grows linearly in $n$.
\end{proof}

This proves point (2) of \cref{thm: main}.

It is easy to see that the class of sequences in this theorem contains infinitely many nonlinear sequences. 

\subsection{Infinite Sequences}
We can also consider $u$-saturated \textit{infinite} sequences. In other words, let $s$ be an infinite sequence on some alphabet $A$ (which can be infinite). We say $s$ is $u$-saturated if $s$ is $r$-sparse (defined in the same way as in the finite-length case), $s$ does not contain a copy of $u$ as a subsequence, and inserting a new letter between two letters of $s$ either violates $r$-sparsity or induces a copy of $u$.

In the following theorem, we prove that for many sequences $u$, there exists a doubly infinite sequence $s$ on an infinite alphabet $A$ such that $s$ is $u$-saturated.

Say that $u$ is \textbf{strongly irreducible} if for any two letters $\alpha,\beta$ in $u$, it is not the case that all the $\alpha$'s occur before all the $\beta$'s or vice versa. Note that this implies the notion of irreducibility introduced earlier.
\begin{thm}\label{thm: infinite sequences}
    Let $u$ be a strongly irreducible sequence on $r$ letters that contains each of its letters more than once. Then there exists a doubly infinite sequence on the alphabet $\{a_1,\dots,a_{r-1}\}\cup \Z$ that is $u$-saturated.
\end{thm}
\begin{proof}
    The idea is similar to the proof of \cref{thm: aa...bb}. Define $t_r(k)$ to be the infinite sequence over the alphabet $\{a_1,\dots, a_{r-1},x\}$ given by
    \[
    \dots (a_1\dots a_{r-1})(a_1\dots a_{r-1})(a_1\dots a_{r-1}x)^k (a_1\dots a_{r-1})(a_1\dots a_{r-1})\dots .
    \]
    Since $u$ contains each of its letters at least twice, the sequence $t_r(1)$ avoids $u$. Moreover, for sufficiently large $k$, it is clear that $t_r(k)$ contains a copy of $u$. Let $p$ be the maximum integer such that $t_r(p)$ avoids $u$. Note that $p\geq 1$.  
    
    For $n\in \mathbb Z$, define
    \[
    t[n] = (a_1\dots a_{r-1}n)^p.
    \]
    We then set
    \[
    s = \dots t[-1]t[0]t[1]t[2]\ldots = \bigcup_{n=-\infty}^\infty t[n].
    \]
    We claim that $s$ is $u$-saturated. It is clearly $r$-sparse.
    
    \smallskip \noindent\textit{Avoidance.} Suppose $s$ contains a copy of $u$. Since $u$ is strongly irreducible, such a copy can involve at most one letter from $\mathbb Z$. Thus it must lie on the alphabet $\{a_1,\dots, a_{r-1},n\}$ for some $n\in\mathbb Z$. But the subsequence of $s$ on these letters is isomorphic to $t_r(p)$, which by definition avoids $u$. Hence $s$ avoids $u$.
    
    \smallskip \noindent\textit{Saturation.} We cannot insert any $a_i$ without breaking $r$-sparsity. If we insert an integer $n$ without violating $r$-sparsity, then $s$ acquires a subsequence isomorphic to $t_r(p+1)$ over the alphabet $\{a_1,\dots, a_{r-1},n\}$. Since $p$ was chosen maximally, $t_r(p+1)$ contains a copy of $u$. It follows that $s$ is $u$-saturated.
\end{proof}

\begin{rem}
    Consider the finite analogue of the sequence $s$, namely
    \[
    t[1]t[2]\dots t[n].
    \]
    The same argument shows that inserting a letter between $t[k]$ and $t[n-k]$ necessarily violates either $r$-sparsity or $u$-avoidance, where $k$ is a constant depending only on $u$. Thus this construction yields a sequence that is ``almost $u$-saturated.'' Moreover, if we run \cref{alg: algorithm} starting from this sequence, the contiguous block 
    \[
    t[k]t[k+1]\dots t[n-k]
    \] 
    is preserved. It is natural to ask whether this additional structure can be exploited to prove that the algorithm terminates in $O(n)$ steps, thereby resolving \cref{conj: main} (and a version of \cref{conj: algorithm}) for the class of sequences considered in \cref{thm: infinite sequences}.
\end{rem}

\section{Saturation for Sequences on 3 letters}\label{sec: 3-letter sequences}
In \cite{ANAND2025382}, it was shown that every sequence on two letters has a linear saturation function. The proof provided an explicit construction of a $u$-saturated sequence on $n$ letters for arbitrary $n$, relying heavily on the sequence $s_2(u)$. This construction extends naturally to an arbitrary sequence $u$ on $r$ letters. The same argument shows that the resulting sequence is \textbf{$u$-semisaturated}, i.e., saturated except that it need not avoid $u$. However, when $u$ involves more than two letters, the sequence obtained from this construction may fail to avoid $u$.  

In this section, we restrict attention to the case where $u$ has three letters, and prove that under certain conditions the construction of \cite{ANAND2025382} indeed yields a $u$-saturated sequence. Consequently, we obtain $\Sat(n,u)=O(n)$ in this setting.  

Throughout, let $u$ be a fixed three-letter sequence of length greater than or equal to six of the form
\[
u = abc\dots xyz,
\]
where $a,b,c$ and $x,y,z$ are distinct letters, and $x,y,z$ is a permutation of $\{a,b,c\}$. Define a family of sequences $s[n]$ on $n$ letters recursively as follows. Set $s[3]=s_3(u)$. Given $s[n-1]$, let $x,y$ denote its last two letters, and let $z$ be a new letter not appearing in $s[n-1]$. Construct $s[n]$ by appending to $s[n-1]$ a copy of $s_3(u)$ on the letters $x,y,z$, omitting its initial two letters $x,y$. In particular, the terminal segment of $s[n]$ is a full copy of $s_3(u)$ on $x,y,z$.  

This definition of $s[n]$ is directly analogous to the construction of the $u$-saturated sequence in Theorem~3.10 of \cite{ANAND2025382}.

\begin{prop}
    The sequence $s[n]$ is $u$-semisaturated.
\end{prop}
\begin{proof}
    The proof is analogous to the last part of the proof of Theorem 3.10 in \cite{ANAND2025382}.
\end{proof}
Thus, it remains to determine when $s[n]$ avoids $u$.

We begin by analyzing the length of $s_3 := s_3(u)$. By definition, $s_3 = 123123\dots$. If $s_3'$ denotes $s_3$ with one new letter appended at the end (in a way that preserves $3$-sparsity), then $s_3'$ contains a copy of $u$. Moreover, this copy necessarily uses both the first and last letters of $s_3'$. Hence $a=1$ and $(b,c)=(2,3)$ or $(3,2)$ in this copy.

To compute $|s_3'|$, imagine first writing the copy of $u$, then inserting letters from $\{1,2,3\}$ between consecutive letters so that the resulting sequence has the form $123123\dots$. The number of insertions depends on the consecutive pair:
\begin{itemize}
    \item $0$ insertions for $12,23,31$,
    \item $1$ insertion for $13,21,32$,
    \item $2$ insertions for $11,22,33$.
\end{itemize}

Define
\begin{align*}
    f_0(u) &= \#\{\,\text{consecutive pairs of the form } ab,bc,ca\text{ in $u$}\}, \\
    f_1(u) &= \#\{\,\text{consecutive pairs of the form } ac,ba,cb\text{ in $u$}\}, \\
    f_2(u) &= \#\{\,\text{consecutive equal-letter pairs in } u\,\}.
\end{align*}
Clearly,
\[
f_0(u) + f_1(u) + f_2(u) = |u|-1.
\]
If the copy of $u$ in $s_3'$ uses the letters $1,2,3$ in this order, then
\[
|s_3'| = |u| + f_1(u) + 2f_2(u),
\]
while if it uses $1,3,2$, then
\[
|s_3'| = |u| + f_0(u) + 2f_2(u).
\]
Therefore,
\begin{equation}\label{eqn: length of s_3}
    |s_3| = |u| - 1 + f_2(u) + \min\{f_0(u),f_1(u)\}.
\end{equation}

\begin{thm}\label{thm: 3-letter sequences}
    Let $u=abc\dots xyz$ be a three-letter sequence with $a,b,c$ distinct. Suppose 
    \[
    xyz \in \{abc, bca, cab\}, \qquad f_0(u)\ge f_1(u)+5.
    \]
    Then $s[n]$ avoids $u$, and hence $\Sat(n,u)=O(n)$.
\end{thm}

\begin{proof}
We use a similar argument to that of Theorem 3.10 in \cite{ANAND2025382}, but the 3-letter case is much more complicated (which is also why we need to impose many more conditions). The proof proceeds by induction and a careful case analysis of where a hypothetical copy of $u$ could appear. The parameters $f_0(u)$ and $f_1(u)$ control how much space such a copy would need to take.

We proceed by induction on $n$. The base case holds since $s[3]=s_3$ is $u$-saturated, and therefore avoids $u$. Assume inductively that $s[n-1]$ avoids $u$.

Suppose for contradiction that $s[n]$ contains a copy of $u$. We distinguish two cases.

\smallskip
\noindent
\textsc{Case 1:} One of the letters is $z$.  
It is clear from the conditions that $u$ is strongly irreducible, so the other two letters must be $x$ and $y$, as all other letters appear strictly before the first $z$. The subsequence of $s[3]$ consisting of $x,y$ is of the form
\[
Pxyzxyz\dots,
\]
where the part after $P$ is isomorphic to $s_3$ and $P$ contains only $x,y$. Because the first three letters of $u$ are distinct, any copy of $u$ in this subsequence can use at most two letters before the first $z$, and these must be distinct. Thus it suffices to consider the cases $P=y$:
\[
xyzxyz\dots \quad \text{and} \quad yxzxyz\dots,
\]
with both sequences of length $|s_3|$. The first is isomorphic to $s_3$, hence avoids $u$. In the second, if a copy of $u$ exists, it must use the initial $yx$; otherwise, the same copy could be embedded in the first sequence, a contradiction. Thus the copy is on the alphabet $\{y,x,z\}$, and its length is at least
\[
|u| + f_0(u) + 2f_2(u) - 2,
\]
where we subtract $2$ because no insertions are needed between the initial $yx$ and the following $xz$. Since $f_0(u)\ge f_1(u)+3$, this contradicts \eqref{eqn: length of s_3}.

\smallskip
\noindent
\textsc{Case 2:} None of the letters is $z$.  
Let the letters be $\alpha,\beta,\gamma$ in the order of their first appearance in $s[3]$. The first $\gamma$ appears when extending $s[k-1]$ to $s[k]$ for some $k<n$. Since $u$ is strongly irreducible, some $\alpha$ must appear after some $\gamma$, which is only possible if $\alpha$ is among the last two letters of $s[k-1]$. Similarly, $\beta$ must also be among the last two letters. Without loss of generality, let $\beta$ be the last letter and $\alpha$ the second-last (if $k-1=0$, then $\alpha,\beta,\gamma$ are simply the first three letters of $s[n]$, and the same argument applies). Thus there is a copy of $u$ in
\[
P\alpha\beta\gamma\alpha\beta\gamma\dots Q,
\]
where $P$ contains only $\alpha,\beta$, the middle block has length $|s_3|$, and $Q$ contains only the last two letters of the middle block. We now consider three subcases, depending on the ending of the middle block.

\smallskip
\noindent
\textsc{Case 2.1:} The middle block ends $\alpha\beta\gamma$.  
We must check the following sequences of length $|s_3|$:
\[
\alpha\beta\gamma\dots \alpha\beta\gamma,\quad 
\beta\alpha\gamma\dots \alpha\beta\gamma,\quad 
\alpha\beta\gamma\dots \alpha\gamma\beta,\quad 
\beta\alpha\gamma\dots \alpha\gamma\beta.
\]
The first two follow from Case~1. For the latter two, let $\widetilde{u}$ be the reversal of $u$. Since $xyz\in\{abc,bca,cab\}$, $f_i(\widetilde{u})=f_i(u)$, so $\widetilde{u}$ satisfies the theorem’s hypotheses (note that we must rename the letters $a,b,c$ when computing $f_i(\widetilde{u})$). The reversal of the third sequence is isomorphic to the second sequence of Case~1 for $\widetilde{u}$, hence avoids $\widetilde{u}$, so the third sequence avoids $u$. For the fourth, any copy of $u$ must use both the initial $\beta\alpha$ and the terminal $\gamma\beta$; otherwise it embeds in one of the previous sequences. Thus its length is at least
\[
|u| + f_0(u) + 2f_2(u) - 4.
\]
Since this is $\le |s_3|$ yet $f_0(u)\ge f_1(u)+5$, we again contradict \eqref{eqn: length of s_3}.

\smallskip
\noindent
\textsc{Case 2.2:} The middle block ends $\beta\gamma\alpha$.  
Here the relevant sequences are
\[
\alpha\beta\gamma\dots \alpha\beta\gamma\alpha,\quad
\beta\alpha\gamma\dots \alpha\beta\gamma\alpha,\quad
\alpha\beta\gamma\dots \alpha\beta\alpha\gamma,\quad
\beta\alpha\gamma\dots \alpha\beta\alpha\gamma,
\]
each of length $|s_3|$. The proof is exactly analogous to Case~2.1.

\smallskip
\noindent
\textsc{Case 2.3:} The middle block ends $\gamma\alpha\beta$.  
Here the sequences are
\[
\alpha\beta\gamma\dots \alpha\beta\gamma\alpha\beta,\quad
\beta\alpha\gamma\dots \alpha\beta\gamma\alpha\beta,\quad
\alpha\beta\gamma\dots \alpha\beta\gamma\beta\alpha,\quad
\beta\alpha\gamma\dots \alpha\beta\gamma\beta\alpha,
\]
each of length $|s_3|$. The argument is identical to the previous subcases.

\smallskip

In all cases we obtain a contradiction, so $s[n]$ avoids $u$. Since $|s[n]|$ grows linearly in $n$, it follows that $\Sat(n,u)=O(n)$.
\end{proof}
This proves point (3) of \cref{thm: main}.

\begin{exam}
Some examples of sequences $u$ satisfying the conditions of \cref{thm: 3-letter sequences} are
\[abcacbcabca, abcbabcabc,abcbcbacbabcabcabc.\]
Additionally, all sequences of the form $(abc)^t$, $(abc)^ta$, $(abc)^tab$ satisfy the conditions as long as $t\ge 2$. Thus, \cref{thm: 3-letter sequences} implies \cref{prop: generalized alternations}. Finally, note that for any sequence $u$ on letters $a,b,c$, $(abc)u(abc)^t$ satisfies the conditions for large enough $t$. Thus $\Sat(n,(abc)u(abc)^t)=O(n)$ for large enough $t$.
\end{exam}
% \iffalse
\section{A Linear Program for \texorpdfstring{$\Sat(n,u)$}{Sat(n,u)}}\label{sec: linear program}

In this section, we describe a linear program that computes the exact value of $\Sat(n,u)$. This is similar to the linear program found in \cite{brahms01matrices}, but it is significantly more complicated. Let $u$ be a sequence of length $\ell$ with $r$ distinct letters. Call an $r \times \ell$ 0-1 matrix $P$ a $u$-pattern if $P$ has a single 1 per column, and the sequence obtained by listing the row indices of the ones from left to right is isomorphic to $u$. Call a $r\times (\ell-1)$ 0-1 matrix $Q$ a $u^+$-pattern if every column has exactly one $1$, except for a single column $c$ which has two ones; and the sequence obtained by listing the row indices of the ones from left to right, where the one in $c$ with a higher row index is placed first, is isomorphic to $u$. Define a $u^-$-pattern similarly, except the one in $c$ with a lower row index is placed first. For example, if $u=abcac$ then we have the following examples:\\

% \vspace{5mm}
\begin{center}
    \begin{minipage}{5cm}
        \[\begin{pmatrix}
            & & \bullet& &\bullet\\
            & \bullet& & &\\
            \bullet & & &\bullet& 
        \end{pmatrix}\]

        \vspace{5mm}
        \centering $u$-pattern
    \end{minipage}
    \begin{minipage}{5cm}
        \[\begin{pmatrix}
            & \bullet& &\bullet\\
            & \bullet& & \\
            \bullet & & \bullet& 
        \end{pmatrix}\]

        \vspace{5mm}
        \centering $u^+$-pattern
    \end{minipage}
    \begin{minipage}{5cm}

        \[\begin{pmatrix}
            & & \bullet& \bullet\\
            & \bullet& & \\
            \bullet & & \bullet& 
        \end{pmatrix}\]

        \vspace{5mm}
        \centering $u^-$-pattern
    \end{minipage}
\end{center}

\vspace{3mm}
For a positive integer $m$, let $[m]$ denote the set $\{1,\dots,m\}$.

\begin{thm}\label{thm: linear program}
    Let $N$ be an integer greater than $\Sat(n,   u)$. Define $\mathcal{P}$ to be the set of all index sets $P' \subset [n] \times [N]$ such that $P'$ forms a copy of some $u$-pattern $P$ in the all-ones matrix $\mathbf{1}_{n \times N}$. For each pair $(i, j) \in [n] \times [N]$, let
    \[ A_{i,j} = \{P' \in \mathcal{P} : (i,j) \in P'\} \]
    denote the collection of patterns in $\mathcal{P}$ that include the entry $(i,j)$. Note that $A_{i,j} \neq \varnothing$ for all $(i,j)$, and each pattern $P' \in \mathcal{P}$ satisfies $|P'| = \ell$, where $\ell$ is the size of any $u$-pattern.

    Define $\mathcal{P}^+$, $\mathcal{P}^-$, $A_{i,j}^+$, and $A_{i,j}^-$ analogously for the sets of modified $u$-patterns. Again, each $P' \in \mathcal{P}^\pm$ satisfies $|P'| = \ell$.

    Define variables $x_{i,j}$ for $1\le i\le r$ and $1\le j\le N$ and $y_{P'}$ for each $P'\in \mathcal P$ or $\mathcal P^\pm$. For notational convenience, set \[f(P') :=\sum_{(i,j)\in P'}x_{i,j}.\] Suppose $x_{i,j}$ and $y_{P'}$ satisfy

    \begin{alignat}{2}
        & x_{i,j}\in\{0,1\} &\hspace{1.5cm}& \forall(i,j)\in[n]\times[N] \\
        & y_{P'}\in\{0,1\}  &&       \forall P'\in\mathcal P \\
        & f(P')<\ell  &&       \forall P'\in\mathcal P \label{eqn:not_contain}\\
        & (\ell-1)y_{P'}-f(P') \le0 &&  \forall P'\in\mathcal P\cup \mathcal P^{\pm}  \label{eqn:yp'_cond_1}\\
        & y_{P'}-f(P')\ge-\ell+2 &&  \forall P'\in\mathcal P\cup \mathcal P^{\pm}  \label{eqn:yp'_cond_2}\\
        & \sum_{P'\in A_{i,j}\cup A_{i,j}^+}y_{P'} - \sum_{t=i+1}^{n} x_{t,j}+\sum_{t=j-r+1}^{j+r-2}x_{i,t}\ge 0 &&  \forall(i,j)\in[n]\times[N]  \label{eqn:changeone+above}\\
        & \sum_{P'\in A_{i,j}\cup A_{i,j}^+}y_{P'} - \sum_{t=1}^{i-1} x_{t,j}+ \sum_{t=j-r+2}^{j+r-1}x_{i,t}\ge 0 &&  \forall(i,j)\in[n]\times[N]  \label{eqn:changeone+under} \\
        & \sum_{P'\in A_{i,j}\cup A_{i,j}^-}y_{P'} - \sum_{t=i+1}^{n} x_{t,j}+\sum_{t=j-r+2}^{j+r-1}x_{i,t}\ge 0 &&  \forall(i,j)\in[n]\times[N]  \label{eqn:changeone-above} \\
        & \sum_{P'\in A_{i,j}\cup A_{i,j}^-}y_{P'} - \sum_{t=1}^{i-1} x_{t,j}+ \sum_{t=j-r+1}^{j+r-2}x_{i,t}\ge 0 && \forall(i,j)\in[n]\times[N]   \label{eqn:changeone-under} \\
        & \sum_{i = 1}^n x_{i,j}\le 1 &&  \forall j \in [N]  \label{eqn:onepercol} \\
        & \sum_{i = 1}^n x_{i,j} \ge \sum_{i = 1}^n x_{i,j+1} &&  \forall j \in [N-1]  \label{eqn:leftjustified} \\
        & \sum_{t = j}^{j+r-1} x_{i,t}\le1 && \forall (i,j) \in [n]\times[N-r+1]. \label{eqn:sparsity}\\
        & x_{i,i}=1  && i\in [r-1]\label{eqn: nontrivial}
    \end{alignat}
    
    Then $\displaystyle \min\sum_{i,j}x_{i,j}=\Sat(u, n)$.
\end{thm}
\begin{proof}
    The interpretations for the conditions are as follows: \cref{eqn:not_contain} says that there is no copy of $u$, \cref{eqn:changeone+above,eqn:changeone+under,eqn:changeone-above,eqn:changeone-under} encode the saturation condition, and \cref{eqn:sparsity} encodes the sparsity condition. The other conditions make sure everything is in the right ``syntax''. 
    
    Suppose we have numbers $x_{i,j},y_{P'}$ satisfying the above conditions. Let $X = [x_{ij}]$. Then by conditions \eqref{eqn:onepercol} and \eqref{eqn:leftjustified}, columns $1,\dots, k$ of $X$ have exactly one $1$, and columns $k+1,\dots, N$ are empty, for some $k$. We define the sequence $s$ on $n$ letters with length $k$, by $s_{j}$ being the unique index such that $x_{s_j,j}=1$. We claim that $s$ is a $u$-saturated sequence. First, condition \eqref{eqn:sparsity} implies $s$ is $r$-sparse. Next, condition \eqref{eqn:not_contain} implies there is no copy of $u$ in $s$, so $s$ avoids $u$. Now we need to show that $s$ is $u$-saturated. Consider $y_{P'}$ for $P'\in \mathcal P \cup \mathcal P^{\pm}$. Note that $f(P')\le \ell-1$ for $P'\in \mathcal P^{\pm}$ since $P'$ only has $\ell-1$ nonzero columns. Conditions \eqref{eqn:yp'_cond_1} and \eqref{eqn:yp'_cond_2} then implies $y_{P'}=0$ if $f(P')<\ell-1$ and $y_{P'}=1    $ if $f(P')=\ell-1$, i.e. $P'$ is missing exactly one letter. Since $P'$ has a column with two ones, it must be one of the ones in that column if $y_{P'}=1$. Now, we claim that conditions \eqref{eqn:changeone+above}, \eqref{eqn:changeone+under},\eqref{eqn:changeone-above},\eqref{eqn:changeone-under} together imply that $s$ is $u$-saturated. Together, they correspond to adding a $i$ either before or after $s_j$. If $j>k$, then $x_{t,j}=0$ for all $t$, so the conditions evidently hold. Now consider the case when $j\le k$, so $x_{s_j,j}=1$. If $i=s_j$ then $x_{i,j}=x_{s_j,j}=1$ so all of the conditions are satisfied. Now suppose $i<s_j$. Then $x_{s_j,j}=1$ so $\sum_{t=i+1}^n x_{t,j}=1$, and $\sum_{t=1}^{i-1}x_{t,j}=0$. Thus conditions \eqref{eqn:changeone+under} and \eqref{eqn:changeone-under} are satisfied. We now claim that \eqref{eqn:changeone+above} corresponds to adding the letter $i$ before $s_j$. Note that \[\sum_{t=i+1}^{n} x_{t,j}=1\] implies that one of the following conditions is true:
    \begin{itemize}
    \setlength\itemsep{5pt}
        \item  $y_{P'}\ge 1$ for some $P'\in A_{ij}$, or
        \item $y_{P'}\ge 1$ for some $P'\in A_{ij}^+$, or
        \item $\displaystyle \sum_{t=j-r+1}^{j+r-2}x_{i,t} \ge 1$.
    \end{itemize}
    \begin{table}[ht]
    \centering
    \renewcommand{\arraystretch}{1.2}
    \setlength{\tabcolsep}{10pt}
    \begin{NiceTabular}{ccccc}
        $u$ & &$n$ & $\Sat(n,u)$& $u$-saturated sequence\\
        \toprule
         $aba$&& $2$ & $2$  & $12$\\
         $aba$& &$3$ & $3$  & $123$\\
         $aba$& &$4$ & $4$  & $1234$\\
         \hdottedline
         $abab$& &$2$ & $3$  & $121$\\
         $abab$& &$2$ & $5$  & $12321$\\
         \hdottedline
         $aabb$& &$2$ & $5$  & $12121$\\
         $aabb$& &$3$ & $7$  & $1232123$\\
         \hdottedline
         $aaab$& &$2$ & $5$  & $12121$\\
         $aaab$& &$3$ & $7$  & $1213231$\\
         \hdottedline
         $ababa$& &$2$ & $4$  & $1212$\\
         $ababa$& &$3$ & $7$  & $1212323$\\
         \hdottedline
         $ababb$& &$2$ & $5$  & $12121$\\
         $ababb$& &$3$ & $7$  & $1212313$\\
         \hdottedline
         $ababa$& &$2$ & $4$  & $1212$\\
         $ababa$& &$3$ & $7$  & $1212323$\\
         \hdottedline
         $abca$& &$3$ & $3$  & $123$\\
         $ababa$& &$4$ & $4$  & $1234$\\
         \hdottedline
         $abcb$& &$3$ & $3$  & $1231$\\
         $abcb$& &$4$ & $5$  & $12341$\\
         \hdottedline
         $abcc$& &$3$ & $5$  & $12312$\\
         $abcc$& &$4$ & $6$  & $123142$\\
         \hdottedline
         $abcba$& &$3$ & $3$  & $1231$\\
         $abcba$& &$4$ & $5$  & $12341$\\
         \toprule
         
    \end{NiceTabular}
    \vskip 4pt
    \caption{Exact Values of $\Sat(n,u)$}
    \label{tab: linear program}
\end{table}

    The third is easily seen to be equivalent to violating the sparsity condition if $i$ is inserted before $s_{j}$. The second implies that if we make the $(i,j)$ spot into a $1$, we get a copy of $P'$ for some $P'\in A_{ij}^+$. Since column $j$ is the only one with two nonzero entries, both $(i,j)$ and $(s_j,j)$ must occur in $P$. Then, by listing out the row-indices of this copy from left to right, putting the one for $(i,j)$ before the one for $(s_j,j)$, we get a copy of $u$. Finally, we similarly get that the first implies that an $i$ before the $s_j$ induces a copy of $u$ in $s$. 
    
    Similarly, condition \eqref{eqn:changeone-above} implies that inserting an $i$ after $s_{j}$ induces a copy of $u$, and if $s_j>i$, we have a similar proof. Thus, $s$ is $u$-saturated. Note that \eqref{eqn: nontrivial} implies that $s$ is nontrivial. Furthermore, it is clear that any $u$-saturated sequence has length $\ge \ell -1 \ge r-1$, so by an appropriate renaming of the letters in $s$, we can make \eqref{eqn: nontrivial} be satisfied.

    Likewise, for any $u$-saturated sequence $s$, we can define numbers $x_{i,j}$ by reversing the above process. Then conditions \eqref{eqn:yp'_cond_1} and \eqref{eqn:yp'_cond_2} uniquely determine the $y_{P'}$, and using the fact that $s$ is $u$-saturated, we can simply reverse the above arguments to show that the remaining conditions are satisfied. Since the length of $s$ is equal to $\sum x_{i,j}$, this completes the proof.
\end{proof} 
The choice of the integer $N$ above can be found by trial and error (or by finding some $u$-saturated sequence with length $N$). 

Using the SAGE implementation of this linear program in \cite{kanungocode}, we get the results in \cref{tab: linear program}. Note that when $n=r$ is the number of distinct letters of $u$, we see that $\Sat(r,u)$ is simply the length of $s_r(u)$. This is easy to prove in general, because any $r$-sparse sequence on $[r]$ is isomorphic to a sequence of the form $12\dots r12\dots $.

It is not feasible to use this linear program to compute $\Sat(n,u)$ when $u$ is long or $n$ is large, because that means that $N$ will be large, implying that the number of variables blows up. We wonder if the algorithm can be improved to compute more nontrivial values of $\Sat(n,u)$.

\section*{Acknowledgements}
I am grateful to my mentor, Jesse Geneson, for proposing this project and for his guidance during the preparation of this paper. I thank James Alitzer for collaborating on the linear program in \cref{sec: linear program}. While working on this paper, I was part of MIT PRIMES-USA, a year-long math research program hosted by MIT. I would like to express my gratitude to the PRIMES program and its organizers for making this experience possible. 

\printbibliography

@article{ANAND2025382,
title = {Sequence saturation},
journal = {Discrete Applied Mathematics},
volume = {360},
pages = {382--393},
year = {2025},
issn = {0166-218X},
doi = {https://doi.org/10.1016/j.dam.2024.09.034},
url = {https://www.sciencedirect.com/science/article/pii/S0166218X24004244},
author = { Anand and Jesse Geneson and Suchir Kaustav and Shen-Fu Tsai},
}

@article{PETTIE20111863,
title = {Generalized Davenport–Schinzel sequences and their 0–1 matrix counterparts},
journal = {Journal of Combinatorial Theory, Series A},
volume = {118},
number = {6},
pages = {1863--1895},
year = {2011},
issn = {0097-3165},
doi = {https://doi.org/10.1016/j.jcta.2011.02.011},
url = {https://www.sciencedirect.com/science/article/pii/S0097316511000550},
author = {Seth Pettie},
}

@article{Fulek2020SaturationPA,
  title={Saturation Problems about Forbidden 0-1 Submatrices},
  author={Radoslav Fulek and Bal{\'a}zs Keszegh},
  journal={SIAM J. Discret. Math.},
  year={2020},
  volume={35},
  pages={1964--1977},
  url={https://api.semanticscholar.org/CorpusID:223957157}
}

@article{DAMASDI2021103321,
title = {Saturation problems in the Ramsey theory of graphs, posets and point sets},
journal = {European Journal of Combinatorics},
volume = {95},
pages = {103321},
year = {2021},
issn = {0195-6698},
doi = {https://doi.org/10.1016/j.ejc.2021.103321},
url = {https://www.sciencedirect.com/science/article/pii/S0195669821000135},
author = {Gábor Damásdi and Balázs Keszegh and David Malec and Casey Tompkins and Zhiyu Wang and Oscar Zamora}
}

@article{KászonyiTuzaGraphSaturation,
author = {Kászonyi, László and Tuza, Zs.},
year = {1986},  
pages = {203 --210},
title = {Saturated graphs with minimal number of edges},
volume = {10},
journal = {Journal of Graph Theory},
doi = {10.1002/jgt.3190100209}
}

@article{Ferrara2017TheSN,
  title={The saturation number of induced subposets of the Boolean lattice},
  author={Michael Ferrara and Bill Kay and Lucas Kramer and Ryan R. Martin and Benjamin Reiniger and Heather C. Smith and Eric Sullivan},
  journal={Discret. Math.},
  year={2017},
  volume={340},
  pages={2479--2487},
  url={https://api.semanticscholar.org/CorpusID:207137903}
}

@article{Geneson2020AlmostAP,
  title={Almost all Permutation Matrices have Bounded Saturation Functions},
  author={Jesse  Geneson},
  journal={Electron. J. Comb.},
  year={2020},
  volume={28},
  pages={P2.16},
  url={https://api.semanticscholar.org/CorpusID:229678039}
}

@article{Berendsohn2021AnEC,
  title={An exact characterization of saturation for permutation matrices},
  author={Benjamin Aram Berendsohn},
  journal={Comb. Theory},
  year={2021},
  volume={3},
  url={https://api.semanticscholar.org/CorpusID:233739638},
  pages={\#17},
}

@misc{brahms01matrices,
    title={Saturation of 0-1 matrices},
    author = {Andrew Brahms and Alan Duan and Jesse Geneson and Jacob Greene},
    year = {2025},
    eprint={2503.03193},
    archivePrefix={arXiv},
    primaryClass={math.CO},
    url={https://arxiv.org/abs/2503.03193}, 
}

@article{davenportschinzel65,
    author = {Harold Davenport and Andrzej Schinzel},
    title = {A Combinatorial Problem Connected with Differential Equations},
    journal = {American Journal of Mathematics},
    year = {1965},
    volume = {87},
    number = {3},
    pages = {684--694},
}

@article{Atallah85,
    author = {Mikhail J. Atallah},
    title = {Some dynamic computational geometry problems},
    journal = {Computers and Mathematics with Applications},
    year = {1985},
    volume = {11},
    number = {12},
    pages = {1171--1181},
}

@article{Keszegh09,
    author = {Balz{\'a'}s Keszegh},
    title = {On linear forbidden submatrices},
    journal = {Journal of Combinatorial Theory, Series A},
    year = {2009},
    volume ={116},
    number = {1},
    pages = {232--241}
}

@misc{kanungocode,
    author = {Shihan Kanungo},
    note = {Sequence Saturation Code (2025), online at \texttt{\href{https://github.com/shihankanungo/Exact-Sequence-Saturation}{github}}}
}

\end{document}